\documentclass[a4paper]{article}
\usepackage{amssymb,amsmath,amsfonts,amsthm,graphicx,verbatim}
\usepackage{stmaryrd}
\usepackage{hyperref}

\theoremstyle{plain}
\newtheorem{proposition}{Proposition}[section]
\newtheorem{theorem}[proposition]{Theorem}
\newtheorem{corollary}[proposition]{Corollary}

\newtheorem{lemma}[proposition]{Lemma}

\def\msig{\mc{F}^\sigma_m}
\def\RF{\mathbb{R}\msig}
\def\RFP{\mathbb{R}\mc{F}_m^{\sigma+}}
\def\RFM{\mathbb{R}\mc{F}_m^{\sigma-}}
\def\mc{\mathcal}
\def\prog{\textsf{DensityBounder }}
\def\F{\mathcal{F}}
\def\tr{\textrm}
\def\cH{\char`\#}
\def\cD{\char`\_}

\begin{document}

\title{Hypergraphs do jump}

\author{Rahil Baber\thanks{Department of Mathematics, UCL, London, WC1E 6BT, UK. Email: rahilbaber@hotmail.com.}
\and John Talbot\thanks{Department of Mathematics, UCL, London, WC1E
6BT, UK. Email: talbot@math.ucl.ac.uk.  This author is a Royal
Society University Research Fellow.}}

\date\today

\maketitle

\begin{abstract}

We  say that
$\alpha\in [0,1)$ is a \emph{jump} for an integer $r\geq 2$ if there
exists $c(\alpha)>0$ such that for all $\epsilon >0 $ and all $t\geq
1$ 
any $r$-graph
with $n\geq n_0(\alpha,\epsilon,t)$ vertices and density at least
$\alpha+\epsilon$ 
contains a subgraph on $t$
vertices of density at least $\alpha+c$.

The Erd\H os--Stone--Simonovits theorem \cite{ES2}, \cite{ES1}
implies that for $r=2$ every $\alpha\in [0,1)$ is a jump. Erd\H os
\cite{E} showed that for all $r\geq 3$, every $\alpha\in [0,r!/r^r)$
is a jump. Moreover he made his famous ``jumping constant
conjecture'' that for all $r\geq 3$, every $\alpha \in [0,1)$ is a
jump. Frankl and R\"odl \cite{FR} disproved this conjecture by
giving a sequence of values of non-jumps for all $r\geq 3$. 

We use Razborov's flag algebra method \cite{R4} to show that jumps exist for $r=3$ in the interval $[2/9,1)$. These are the first examples of jumps for any $r\geq 3$ in the
interval $[r!/r^r,1)$. To be precise we show that for $r=3$ every
$\alpha \in [0.2299,0.2316)$ is a jump.

We also give an improved upper bound for the Tur\'an density of $K_4^-=\{123,124,134\}$: $\pi(K_4^-)\leq 0.2871$. This in turn implies that for $r=3$ every $\alpha \in [0.2871,8/27)$ is a jump.
\end{abstract}

\section{Introduction}
An \emph{$r$-uniform hypergraph} (or \emph{$r$-graph} for short) is a pair
$F=(V(F),E(F))$ where $V(F)$ is a set of \emph{vertices} and $E(F)$
is a family of $r$-subsets of $V(F)$ called \emph{edges}. So a $2$-graph is
a simple graph. For ease of notation we often identify an
$r$-graph $F$ with its edge set. The \emph{density} of an $r$-graph $F$ is \[
d(F)=\frac{|E(F)|}{\binom{n}{r}}.\]

We say that $\alpha\in [0,1)$ is a \emph{jump} for an integer $r\geq
2$ if there exists $c(\alpha)>0$ such that for all $\epsilon >0 $
and all $t\geq 1$ there exists $n_0(\alpha,\epsilon,t)$ such that
any $r$-graph with $n\geq n_0(\alpha,\epsilon,t)$ vertices and at
least $(\alpha+\epsilon)\binom{n}{r}$ edges contains a subgraph on
$t$ vertices with at least $(\alpha+c)\binom{t}{r}$ edges.

The Erd\H os--Stone--Simonovits theorem \cite{ES2}, \cite{ES1}
implies that for $r=2$ every $\alpha\in [0,1)$ is a jump. Erd\H os
\cite{E} showed that for all $r\geq 3$, every $\alpha\in [0,r!/r^r)$
is a jump. He went on to make his famous ``jumping constant
conjecture'' that for all $r\geq 3$, every $\alpha \in [0,1)$ is a
jump. Frankl and R\"odl \cite{FR} disproved this conjecture by
giving a sequence of values of non-jumps for all $r\geq 3$. More
recently a number of authors have given more examples of non-jumps
for each $r \geq 3$ in the interval $[5r!/2r^r,1)$ (see \cite{FPRT} for example). However nothing
was previously known regarding the location of jumps or non-jumps in
the interval $[r!/r^r,5r!/2r^r)$ for any $r\geq 3$.

We give the first examples of jumps for any $r\geq 3$ in the
interval $[r!/r^r,1)$.
\begin{theorem}\label{main:thm}
If $\alpha\in [0.2299,0.2316)$ then $\alpha$ is a jump for $r=3$.
\end{theorem}
In order to explain our proof we require some definitions and a
theorem of Frankl and R\"odl \cite{FR}.

Let $F$ be an $r$-graph with vertex set $[n]=\{1,2,\ldots,n\}$ and
edge set $E(F)$. Define \[ S_n=\{(x_1,\ldots,x_n)\in \mathbb{R}^n:
\sum_{i=1}^n x_i=1,x_i\geq 0\}.\] For $x\in S_n$ let
\[
\lambda(F,x)=\sum_{\{i_1,i_2,\ldots,i_r\}\in E(F)}r!x_{i_1}x_{i_2}\cdots x_{i_r}.
\]
The \emph{Lagrangian} of $F$ is defined to be
\[
\lambda(F)=\max_{x\in S_n}\lambda(F,x).
\]

Given a family of $r$-graphs $\F$ we say that an $r$-graph $H$ is
\emph{$\F$-free} if $H$ does not contain a subgraph isomorphic to
any member of $\F$. For any integer $n\geq 1$ we define the
\emph{Tur\'an number} of $\F$ to be
\[
\tr{ex}(n,\F)=\max\{|E(H)|:H \tr{ is $\F$-free},\ |V(H)|=n\}.
\]
The \emph{Tur\'an density} of $\F$ is defined to be the following
limit (a simple averaging argument shows that it always exists)
\[
\pi(\F)=\lim_{n\to\infty} \frac{\tr{ex}(n,\F)}{\binom{n}{r}}.\]
We say that $\alpha$ is \emph{threshold} for $\F$ if $\pi(\F)\leq \alpha$.
\begin{theorem}[Frankl and R\"odl \cite{FR}]\label{FR:thm}
  The following are equivalent:
\begin{itemize}
\item[(i)] $\alpha$ is a jump for $r$.
\item[(ii)] $\alpha$ is threshold for a finite family $\F$ of $r$-graphs satisfying \[\min_{F\in
\F}\lambda(F)>\alpha.\]
\end{itemize}
\end{theorem}
Let $F_r$ be the $r$-graph consisting of a single edge. Since any
$\alpha\in [0,1)$ is threshold for $F_r$ and $\lambda(F_r)=r!/r^r$,
Theorem \ref{FR:thm} trivially implies Erd\H os's result \cite{E}
that for each $r\geq 3$, every $\alpha\in [0,r!/r^r)$ is a jump for $r$.

The original version of Erd\H os's jumping constant conjecture
asserted that $r!/r^r$ is a jump for every $r\geq 3$. This fascinating
problem is still open, even for $r=3$. Erd\H os speculated \cite{E}
that $3!/3^3=2/9$ was threshold for the following family of 3-graphs
$\F^*=\{F_1,F_2,F_3\}$, where
\[
F_1 = \{123, 124, 134\}, F_2 = \{123, 124, 125, 345\}, F_3 = \{123,
124, 235, 145, 345\}.\] It is straightforward to check that
$\lambda(F_1)=8/27,$ $\lambda(F_2)=\frac{189+15\sqrt{5}}{961}$ and
$\lambda(F_3)=6/25$. Since $\min_{1\leq i\leq
3}\lambda(F_i)=\lambda(F_2)>2/9$, if $2/9$ were threshold for $\F^*$
then Theorem \ref{FR:thm} would imply $2/9$ is a jump for $r=3$.

Unfortunately Erd\H os's suggestion is incorrect: $2/9$ is not
threshold for $\F^*$. There exist 7 vertex 3-graphs that are
$\F^*$-free with Lagrangians greater than $2/9$. By taking
appropriate ``blow-ups'' of such 3-graphs we find that
$\pi(\F^*)>2/9$. (To be precise we could take blow-ups of $F_4$,
defined below, to show that $\pi(\F^*)\geq  0.2319$.) However Erd\H
os's idea suggests a natural approach to proving that $2/9$ is a
jump for $r=3$. Let $\F'$ be a family of $3$-graphs containing
$F_1,F_2,F_3$ with the property that $\min_{F\in\F'}\lambda(F)>2/9$.
If we can show that $2/9$ is threshold for $\F'$ then (by Theorem
\ref{FR:thm}) $2/9$ is a jump for $r=3$.

A search of all $3$-graphs with at most 7 vertices  yields the
following two additional 3-graphs which we can add to $\F'$
\[
F_4 = \{123, 135, 145, 245, 126, 246, 346, 356, 237, 147, 347, 257,167\},
\]
\[
F_5 = \{123, 124, 135, 145, 236, 346, 256, 456, 247, 347, 257,
357,167\}.
\]
It is easy to check that $\lambda(F_4)\geq 0.2319>\lambda(F_2)$ (to
see this set $x_1=x_2=x_3=0.164, x_4=0.154, x_5=x_6=x_7=0.118$) and
$\lambda(F_5)\geq \lambda(F_2)$ (set $\mu=\frac{18-3\sqrt{5}}{31}$,
$x_1=x_6=x_7=\mu/3, x_2=x_3=x_4=x_5=(1-\mu)/4$).

We can now ask: is it true that $2/9$ is threshold for
$\F'=\{F_1,F_2,F_3,F_4,F_5\}$? Unfortunately this is still false,
there exist 3-graphs on 8 vertices avoiding all members of $\F'$ and
with Lagrangians greater than $2/9$. By taking appropriate
``blow-ups'' of such 3-graphs we can show that $\pi(\F')>2/9$.
Moreover, by considering 8 vertex 3-graphs, numerical evidence
suggests that if $2/9$ is a jump then the size of the jump is
extremely small: $c(2/9)\leq 0.00009254$.

However, although $2/9$ is not threshold for $\F'$ we can
show the following upper bound on the Tur\'an density of $\F'$.
\begin{lemma}\label{main:lem}
The Tur\'an density of $\F'$ satisfies $\pi(\F')\leq
0.2299$.\end{lemma}

Since $0.2299<\min_{F\in \F'}\lambda(F)=\lambda(F_2)=0.2316$,
Theorem \ref{main:thm} is an immediate corollary of Lemma
\ref{main:lem} and Theorem \ref{FR:thm}.

It remains to prove Lemma \ref{main:lem}. For this we make use of
recent work of Razborov \cite{R4} on flag algebras that introduces a
new technique that drastically improves our ability to compute (and
approximate) Tur\'an densities. We outline the necessary background
in the next section but emphasize that the reader should consult
Razborov \cite{RF} and \cite{R4} for a full description of his work.

\section{Computing Tur\'an densities via flag algebras}
\subsection{Razborov's method}
Let $\F$ be a family of $r$-graphs whose Tur\'an density we wish to
compute (or at least approximate). Razborov \cite{R4}, describes a
method for attacking this problem that can be thought of as a
general application of Cauchy--Schwarz using the information given
by small $\F$-free $r$-graphs.

Let $\mc{H}$ be the family of all $\F$-free $r$-graphs of order $l$,
up to isomorphism. If $l$ is sufficiently small we can explicitly
determine $\mc{H}$ (by computer search if necessary).

For $H\in \mc{H}$ and a large $\F$-free $r$-graph $G$, we define
$p(H;G)$ to be the probability that a random $l$-set from $V(G)$
induces a subgraph isomorphic to $H$. Trivially, the density of $G$
is equal to the probability that a random $r$-set from $V(G)$ forms
an edge in $G$. Thus, averaging over $l$-sets in $V(G)$, we can
express the density of $G$ as
\begin{align}\label{EQ:DensityOfG}
d(G) = \sum_{H\in \mc{H}}d(H)p(H;G),
\end{align}
and hence $d(G)\leq \max_{H\in\mc{H}}d(H)$.

This ``averaging'' bound on $d(G)$ is in general rather poor:
clearly it could only be sharp if all subgraphs of $G$ of order $l$
are as dense as possible. It also fails to consider how different
subgraphs of $G$ can overlap. Razborov's flag algebras method allows
us to make use of the information given by examining overlapping
subgraphs of $G$ to give far stronger bounds.

A \emph{flag}, $F=(G_F,\theta)$, is an $r$-graph $G_F$ together with an
injective map $\theta: [s]\to V(G_F)$. If $\theta$ is bijective (and
so $|V(G_F)|=s$) we call the flag a \emph{type}. For ease of notation
given a flag $F=(G_F,\theta)$ we define its order $|F|$ to be
$|V(G_F)|$.

Given a type $\sigma$ we call a flag $F=(G_F,\theta)$ a
$\sigma$-\emph{flag} if the induced labelled subgraph of $G_F$ given
by $\theta$ is $\sigma$. A flag $F=(G_F,\theta)$ is \emph{admissible}
if $G_F$ is $\F$-free.

Fix a type $\sigma$ and an integer $m\leq (l+|\sigma|)/2$. (The
bound on $m$ ensures that an $l$-vertex $r$-graph can contain two
$m$-vertex subgraphs overlapping in $|\sigma|$ vertices.) Let
$\msig$ be the set of all admissible $\sigma$-flags of order $m$, up
to isomorphism. Let $\Theta$ be the set of all injective functions
from $[|\sigma|]$ to $V(G)$. Given $F\in\msig$ and $\theta\in\Theta$
we define $p(F,\theta;G)$ to be the probability that an $m$-set $V'$
chosen uniformly at random from $V(G)$ subject to
$\text{im}(\theta)\subseteq V'$, induces a $\sigma$-flag
$(G[V'],\theta)$ that is isomorphic to $F$.

If $F_a,F_b\in \F_m^\sigma$ and $\theta\in \Theta$ then
$p(F_a,\theta;G)p(F_b,\theta;G)$ is the probability that two
$m$-sets $V_a,V_b\subseteq V(G)$, chosen independently at random
subject to $\text{im}(\theta)\subseteq V_a\cap V_b$, induce
$\sigma$-flags $(G[V_a],\theta)$, $(G[V_b],\theta)$ that are
isomorphic to $F_a, F_b$ respectively. We define a related
probability, $p(F_a,F_b,\theta ; G)$, to be the probability that if
we choose a random $m$-set $V_a\subseteq V(G)$, subject to
$\tr{im}(\theta)\subseteq V_a$ and then choose a random $m$-set
$V_b\subseteq V(G)$ such that $V_a\cap V_b=\tr{im}(\theta)$ then
$(G[V_a],\theta)$, $(G[V_b],\theta)$ are isomorphic to $F_a, F_b$
respectively. Note that the difference between
$p(F_a,\theta;G)p(F_b,\theta;G)$ and $p(F_a,F_b,\theta ; G)$ is due
to the effect of sampling \emph{with} or \emph{without} replacement.
When $G$ is large this difference will be negligible, as the
following lemma tells us. (This is a very special case of Lemma 2.3
in \cite{RF}.)
\begin{lemma}[Razborov \cite{RF}]\label{o(1):eq}
For any $F_a,F_b\in\msig$, and $\theta\in\Theta$,
\[
p(F_a,\theta;G)p(F_b,\theta;G) = p(F_a,F_b,\theta;G)+o(1),
\]
where the $o(1)$ term tends to $0$ as $|V(G)|$ tends to infinity.
\end{lemma}
\begin{proof}
Choose random $m$-sets $V_a, V_b\subseteq V(G)$, independently,
subject to $\tr{im}(\theta)\subseteq V_a\cap V_b$. Let $E$ be the
event that $V_a\cap V_b= \tr{im}(\theta)$. Then
\[
p(F_a,F_b,\theta;G)\mathbf{P}[E]\leq
p(F_a,\theta;G)p(F_b,\theta;G)\leq
p(F_a,F_b,\theta;G)\mathbf{P}[E]+\mathbf{P}[\bar{E}].\] If
$|V(G)|=n$ then
\[
\mathbf{P}[E]=\frac{\binom{n-|\sigma|}{m-|\sigma|}\binom{n-m}{m-|\sigma|}}{\binom{n-|\sigma|}{m-|\sigma|}^2}=1-o(1).\]
\end{proof}


Averaging over a uniformly random choice of $\theta\in \Theta$ we have
\begin{equation}\label{a1:eq}
\mathbf{E}_{\theta\in\Theta}\left[p(F_a,\theta;G)p(F_b,\theta;G)\right]
= \mathbf{E}_{\theta\in\Theta}\left[p(F_a,F_b,\theta;G)\right]+o(1)
\end{equation}

Note that this expectation can be computed by averaging over
$l$-vertex subgraphs of $G$. For an $l$-vertex subgraph $H\in
\mc{H}$ let $\Theta_H$ be the set of all injective maps
$\theta:[|\sigma|]\to V(H)$. Recall that, for $H\in\mc{H}$, $p(H;G)$
is the probability that a random $l$-set from $V(G)$ induces a
subgraph isomorphic to $H$. Thus,
\begin{equation}\label{a2:eq}
\mathbf{E}_{\theta\in\Theta}\left[p(F_a,F_b,\theta;G)\right] =
\sum_{H\in\mc{H}}\mathbf{E}_{\theta\in\Theta_H}\left[p(F_a,F_b,\theta;H)\right]p(H;G).
\end{equation}

Consider a positive semidefinite matrix $Q=(q_{ab})$ of dimension
$|\mc{F}^\sigma_m|$. For $\theta\in \Theta$ define
$\mathbf{p}_\theta=(p(F,\theta;G):F\in \mc{F}^\sigma_m)$. Using
(\ref{a1:eq}), (\ref{a2:eq}) and linearity of expectation we have
\begin{equation}\label{q1:eq}
\mathbf{E}_{\theta\in\Theta}[\mathbf{p}_\theta^TQ\mathbf{p}_\theta]=\sum_{F_a,F_b\in \mc{F}^\sigma_m}\sum_{H\in \mc{H}}q_{ab}\mathbf{E}_{\theta\in \Theta_H}[p(F_a,F_b,\theta;H)]p(H;G)+o(1)\
\end{equation}
For $H\in \mc{H}$ define the coefficient of $p(H;G)$ in (\ref{q1:eq}) by
\begin{equation}\label{ch:eq}
c_H(\sigma,m,Q)=\sum_{F_a,F_b\in
\mc{F}^\sigma_m}q_{ab}\mathbf{E}_{\theta\in
\Theta_H}[p(F_a,F_b,\theta;H)].\end{equation} Suppose we have $t$
choices of $(\sigma_i,m_i,Q_i)$, where each $\sigma_i$ is a type,
each $m_i\leq(l+|\sigma_i|)/2$ is an integer, and each  $Q_i$ is a
positive semidefinite matrix of dimension $|\F_{m_i}^{\sigma_i}|$.
 For $H\in \mc{H}$ define
\[
c_H=\sum_{i=1}^t c_H(\sigma_i,m_i,Q_i).\]
Note that $c_H$ is independent of $G$.

Since each $Q_i$ is positive semidefinite (\ref{q1:eq}) implies that
\[
\sum_{H\in \mc{H}}c_Hp(H;G)+o(1)\geq 0.\]
Thus, using (\ref{EQ:DensityOfG}), we have
\[
d(G)\leq \sum_{H\in \mc{H}}(d(H)+c_H)p(H;G)+o(1).\]
Hence the Tur\'an density satisfies
\begin{equation}\label{newbound:eq}
\pi(\mc{F})\leq \max_{H\in \mc{H}}(d(H)+c_H).\end{equation} Since
the $c_H$ may be negative, for an appropriate choice of the
$(\sigma_i,m_i,Q_i)$, this bound may be significantly better than
the trivial averaging bound given by (\ref{EQ:DensityOfG}).

Note that we now have a semidefinite programming problem: given any
particular choice of the $(\sigma_i,m_i)$ find positive semidefinite
matrices $Q_i$ so as to minimize the bound for $\pi(\F)$ given by
(\ref{newbound:eq}).

\subsection{An example}

\begin{figure}[tbp]
\begin{center}
\includegraphics[height=4cm]{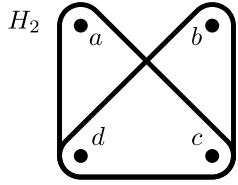}
\caption{The $3$-graph $H_2$, with vertices labelled $a,b,c,d$. Its
two edges are $acd$ and $bcd$.}\label{Fig:H2}
\end{center}
\end{figure}

We now illustrate Razborov's method with a simple example. Let
$K_4^-=\{123,124,134\}$. We will reprove De Caen's  \cite{deC} bound:
$\pi(K_4^-)\leq 1/3$.

Let $l=4$, so $\mc{H}$ consists of all $K_4^-$-free $3$-graphs of
order 4, up to isomorphism. There are three such 3-graphs which we
will refer to as $H_0, H_1,$ and $H_2$, they have $0, 1,$ and $2$
edges respectively (this is enough information to uniquely identify
them). We will use a single type: $\sigma=(G_\sigma,\theta)$ where
$V(G_\sigma)=[2]$, $E(G_\sigma)=\emptyset$ and $\theta(x)=x$. Taking $m=3$, there
are only two admissible $\sigma$-flags of order $3$ up to
isomorphism: $F_0$ and $F_1$, containing $0$ and $1$ edge
respectively.

In order to calculate the coefficients $c_H$ we need to compute
$\mathbf{E}_{\theta\in \Theta_H}[p(F_a,F_b,\theta;H)]$, for each
$H\in \{H_0,H_1,H_2\}$ and each pair $F_a,F_b\in \{F_0,F_1\}$. Their
values are given in the following table.

\begin{center}
\begin{tabular}{c|c c c}
& $H_0$ & $H_1$ & $H_2$\\
\hline
$F_0,F_0$& $1 $& $1/2$& $1/6$\\
$F_0,F_1$ & $0$ & $1/4$ & $1/3$ \\
$F_1,F_1$ & $0 $& $0$ & $1/6$\end{tabular}
\end{center}

As an example of how these values are computed consider
$\mathbf{E}_{\theta \in \Theta_{H_2}}[ p(F_0,F_1,\theta;H_2)]$. This
is the probability that a random choice of $\theta\in \Theta_{H_2}$
and $3$-sets $V_0, V_1\subset V(H_2)$ such that $V_0\cap
V_1=\text{im}(\theta)$, induce $\sigma$-flags $(H_2[V_0],\theta),
(H_2[V_1],\theta)$ that are isomorphic to $F_0, F_1$ respectively. A
random of choice of $\theta\in \Theta_{H_2}$ is equivalent to
picking a random ordered pair of vertices $(u,v)$ from $H_2$, and
setting $\theta(1)=u$ and $\theta(2)=v$. To form the random $3$-sets
$V_0, V_1$ we pick the remaining two vertices of
$V(H_2)\setminus\{u,v\}$ randomly in the order $x,y$ and set
$V_0=\{u,v,x\}, V_1=\{u,v,y\}$. The $\sigma$-flags
$(H_2[V_0],\theta), (H_2[V_1],\theta)$ are isomorphic to $F_0, F_1$
if and only if $V_0\notin E(H_2)$ and $V_1\in E(H_2)$ respectively.
Consequently $\mathbf{E}_{\theta \in \Theta_{H_2}}[
p(F_0,F_1,\theta;H_2)]$ is the probability that a random permutation
$(u,v,x,y)$ of $V(H_2)$ satisfies $\{u,v,x\}\notin E(H_2)$ and
$\{u,v,y\}\in E(H_2)$. Of the $24$ permutations of
$V(H_2)=\{a,b,c,d\}$, see Figure \ref{Fig:H2}, the following $8$
have this property:
\begin{center}
\begin{tabular}{llll}
$(a,c,b,d)$, & $(a,d,b,c)$, & $(b,c,a,d)$, & $(b,d,a,c)$,\\
$(c,a,b,d)$, & $(d,a,b,c)$, & $(c,b,a,d)$, & $(d,b,a,c)$.
\end{tabular}
\end{center}
Hence $\mathbf{E}_{\theta \in \Theta_{H_2}}[
p(F_0,F_1,\theta;H_2)]=8/24=1/3$.

We now need to find a positive semidefinite matrix \[
Q=\begin{pmatrix}
q_{00} & q_{01}\\
q_{01} & q_{11}\end{pmatrix},
\]
to minimize the bound given by (\ref{newbound:eq}). Note that
\begin{eqnarray*}
c_{H_0}&=&q_{00},\\
c_{H_1}&=&\frac{1}{2}q_{00}+\frac{1}{2}q_{01},\\
c_{H_2}&=&\frac{1}{6}q_{00}+\frac{2}{3}q_{01}+\frac{1}{6}q_{11}.\end{eqnarray*}
The bound on $\pi(K_4^-)$ given by (\ref{newbound:eq}) is now
\[
\pi(K_4^-)\leq \max\left\{q_{00},\;\;\;
\frac{q_{00}}{2}+\frac{q_{01}}{2}+\frac{1}{4},\;\;\;
\frac{q_{00}}{6}+\frac{2q_{01}}{3}+\frac{q_{11}}{6}+\frac{1}{2}\right\}.\]
This can be expressed as a semidefinite programming problem. The
solution to which is
\[
Q =\frac{1}{3}
\begin{pmatrix}
 1 & -2\\
-2 & 4
\end{pmatrix}.
\]
Consequently $\pi(K_4^-)\leq \max\{1/3,1/12 ,1/3\}= 1/3$.

\subsection{Proof of Lemma \ref{main:lem}}

To prove $\pi(\F')\leq 0.2299$, we use Razborov's flag algebras
method as outlined above. We set $l=7$, so $\mc{H}$ consists of all $7$ vertex
$3$-graphs that do not contain any $F\in \F'$, up to isomorphism.
There are $4042$ such 3-graphs, which are explicitly determined by the
C++ program \prog (this can be downloaded from \url{http://www.ucl.ac.uk/~ucahjmt/SolnFiles.zip}). To calculate the coefficients
$c_H$ we take six choices of $(\sigma_i,m_i,Q_i)$. The types are
$\sigma_i=((V_i,E_i),\theta_i)$, where
\begin{align*}
V_1 &=[1], & E_1&=\emptyset,\\
V_2 &=[3], & E_2&=\emptyset,\\
V_3 &=[3], & E_3&=\{123\},\\
V_4 &=[5], & E_4&=\{123, 124, 135\},\\
V_5 &=[5], & E_5&=\{123, 124, 345\},\\
V_6 &=[5], & E_6&=\{123, 124, 135, 245\},
\end{align*}
and $\theta_i: [|V_i|]\to V_i,$ maps $x\mapsto x$. Ideally we would use all types of size at most $l-2=5$, however this yields a computationally intractable semidefinite program. Our actual choice was made by experiment, in each case taking the value of $m_i=\lfloor(7+|\sigma_i|)/2\rfloor$. \prog determines the positive semidefinite
matrices $Q_i$ by creating a semidefinite programming problem.
Several implementations of semidefinite program solvers exist. We
chose the CSDP library \cite{B} to solve the problem. The CSDP
library uses floating point arithmetic which may introduce rounding
errors. \prog takes the output of the CSDP program and uses it to
construct the $Q_i$ (removing any rounding errors). Our results can however be  verified without needing to solve a semidefinite program:
\prog can load pre-computed matrices $Q_i$ from the file
\texttt{HypergraphsDoJump.soln} which can also be downloaded from
\url{http://www.ucl.ac.uk/~ucahjmt/SolnFiles.zip}

For each $H\in \mathcal{H}$, $d(H)$ and $c_H$ are calculated by
\prog and using (\ref{newbound:eq}) it computes that $0.2299$ is an
upper bound for $\pi(\F')$. Note that although floating point
operations are used by the semidefinite program solver, our final
computer proof consists of positive semidefinite matrices with
rational coefficients and our proof can be verified using only
integer operations, thus there is no issue of numerical accuracy.

\subsection{Other results}
The program \prog can be used to calculate upper bounds on the
Tur\'an density of other families of $3$-graphs. In particular we
have used it to reproduce Razborov's bound: $\pi(K_4^{(3)})\leq
0.561666$ \cite{R4}.

The conjectured value of $\pi(K_4^-)$ is $2/7=0.2857$. Razborov
\cite{R4} showed that $\pi(K_4^-)\leq 0.2978$. Using \prog we obtain
a new upper bound of $0.2871$ by taking $l=7$ and considering the
following four types $\sigma_i=((V_i,E_i),\theta_i)$ with the given
values of $m_i$ (in each case $\theta_i$ is the identity map):
\begin{center}
\begin{tabular}{lll}
$V_1 = [3]$, & $E_1=\emptyset$, & $m_1 = 5$,\\
$V_2 = [3]$, & $E_2=\{123\}$, & $m_2 = 5$,\\
$V_3 = [4]$, & $E_3=\{123\}$, & $m_3 = 5$,\\
$V_4 = [5]$, & $E_4=\{123,124,125\}$, & $m_4 = 6$.
\end{tabular}
\end{center}
As before the
positive semidefinite matrices $Q_i$ are determined by solving a
semidefinite programming problem.
\begin{theorem}\label{k4:thm}
Let $K_4^-$be the $3$-graph on four vertices with three edges. The
Tur\'an density of $K_4^-$ satisfies
\[
0.2857\ldots=\frac{2}{7} \leq \pi(K_4^-)\leq 0.2871.\]
\end{theorem}
As with our main result our computations can be verified without any
floating point operations so there is no issue of numerical accuracy
in these results. Theorem \ref{k4:thm} yields a second new interval of jumps for $r=3$.
\begin{corollary}\label{d:cor}
If $\alpha\in [0.2871,8/27)$ then $\alpha$ is a jump for $r=3$.
\end{corollary}
\begin{proof} Since $\lambda(K_4^-)=8/27$, this follows directly from Theorem \ref{k4:thm} and Theorem \ref{FR:thm}.\end{proof}

\subsection{Solving the semidefinite program}
Razborov's method as outlined above reduces
the problem of computing an upper bound on a Tur\'an density to
solving a semidefinite programming problem. In practice this may be
computationally difficult. Razborov \cite{R4} describes a number of
ways that this problem can be simplified so as to make the
computation more tractable. We outline one of these ideas below,
which we made use of in our work.

For a type $\sigma$ and the collection of all admissible
$\sigma$-flags of order $m$, $\F_m^\sigma$ define
$\mathbb{R}\F_m^\sigma$ to be the real vector space of formal linear
combinations of $\sigma$-flags of order $m$. Let $\mc{H}$ be the
collection of all admissible $r$-graphs of order $l$.

Let us introduce Razborov's $\llbracket\cdot\rrbracket_\sigma$
notation (which will make our expressions easier to read). Define
$\llbracket\cdot\rrbracket_\sigma:\mathbb{R}\mc{F}^\sigma_m\times\mathbb{R}\F_m^\sigma\to
\mathbb{R}^{|\mc{H}|}$,  by
\[
\llbracket F_a F_b\rrbracket_{\sigma}
=(\mathbf{E}_{\theta\in\Theta_H}[p(F_a,F_b,\theta;H)] : H\in\mc{H}),
\]
for $F_a,F_b\in \F_m^\sigma$ and extend to be bilinear.

For a positive semidefinite matrix $Q$ and $\mathbf{p}=(F :F\in
\mc{F}^\sigma_m)$, the vector of all admissible $\sigma$-flags (in an arbitrary but fixed order), we
have
\[
\llbracket \mathbf{p}^TQ\mathbf{p}\rrbracket_{\sigma} =
(c_H(\sigma,m,Q) : H\in\mc{H}),
\]
where the $c_H$ are as defined in (\ref{ch:eq}).

Razborov \cite{R4} describes a natural change of basis for $\RF$.
The important property (in terms of reducing the computational
complexity of the associated semidefinite program) is that the new
basis is of the form  $\mc{B}=\mc{B}^+\dot{\cup} \mc{B}^-$ and for
all $B^+\in \mc{B}^+$ and $B^-\in \mc{B}^-$ we have $\llbracket
B^+B^-\rrbracket_\sigma=\mathbf{0}$. Thus in our new basis the
corresponding semidefinite program has a solution $Q'$ which is a
block diagonal matrix with two blocks: of sizes $|\mc{B}^+|$ and
$|\mc{B}^-|$ respectively. Since the best algorithms for solving
semidefinite programs scale like the square of the size of block
matrices this change of basis can potentially simplify our
computation significantly.


For a type $\sigma=(G_\sigma,\theta_\sigma)$ we construct the basis
$\mc{B}$ as follows. First construct $\Gamma_\sigma$, the
automorphism group of $\sigma$, whose elements are bijective maps
$\alpha: [|\sigma|]\to [|\sigma|]$ such that
$(G_\sigma,\theta_\sigma\alpha)$ is isomorphic to $\sigma$. The
elements of $\Gamma_\sigma$ act on $\sigma$-flags in an obvious way:
for $\alpha\in \Gamma_\sigma$ and $\sigma$-flag $F=(G_F,\theta_F)$
we define $F\alpha$ to be the $\sigma$-flag $(G_F,\theta_F\alpha)$.
Define subspaces\[ \RFP =\{L\in \RF: L\alpha=L\ \forall \alpha\in
\Gamma_\sigma\}\] and
\[
\RFM=\{L\in \RF: \sum_{\alpha\in
\Gamma_\sigma}L\alpha=\mathbf{0}\}.\] Below we describe how to find
bases $\mc{B}^+,\mc{B}^-$ for these subspaces. By the construction
of these bases it will be clear that $\RF=\RFP\oplus \RFM$. Finally
we will verify that for all $B^+\in \mc{B}^+$ and $B^-\in \mc{B}^-$
we have $\llbracket B^+B^-\rrbracket_\sigma=\mathbf{0}$.

We start with the canonical basis for $\RF$ given by
$\msig=\{F_1,F_2,\ldots,F_t\}$. For each $F_i\in \msig$ define the
orbit of $F_i$ under $\Gamma_\sigma$ by
\[
F_i\Gamma_\sigma=\{F\alpha:\alpha\in \Gamma_\sigma\}.\] Any two orbits are
either equal or disjoint. Suppose there are $u$ distinct orbits:
$O_1,\ldots,O_u$. For $i\in [u]$ let $B_i^+=\sum_{F\in O_i}F$. Then
$\mc{B}^+=\{B_1^+,\ldots,B_u^+\}$ is easily seen to be a basis for
$\RFP$. Moreover if $O_i=\{F_{i_1},\ldots,F_{i_q}\}$ then
$F_{i_1}-F_{i_z}\in \RFM$ for $2\leq z \leq q$ and the union of all
such vectors forms a basis $\mc{B}^-$ for $\RFM$.

We now need to check that if $B^+\in \mc{B}^+$ and $B^-\in \mc{B}^-$
then $\llbracket B^+B^-\rrbracket_\sigma=\mathbf{0}$. If
$B^-\in\mc{B}^-$ then by construction $B^-=F_b\alpha-F_b$ for some
$F_b\in\msig$ and $\alpha\in \Gamma_\sigma$. Moreover $B^+\alpha=B^+$.
Hence, by linearity, \[ \llbracket
B^+B^-\rrbracket_{\sigma}=\llbracket
B^+(F_b\alpha-F_b)\rrbracket_{\sigma}=\llbracket
(B^+\alpha)(F_b\alpha)-B^+F_b\rrbracket_{\sigma}.\] We observe that for any $F_a\in \msig$
\begin{align*}
\llbracket (F_a\alpha)(F_b\alpha)\rrbracket_{\sigma}
&=(\mathbf{E}_{\theta\in\Theta_H}[p(F_a,F_b,\theta\alpha^{-1};H)] :
H\in\mc{H})\\
&=(\mathbf{E}_{\theta\in\Theta_H\alpha^{-1}}[p(F_a,F_b,\theta;H)] :
H\in\mc{H})
\end{align*}
where $\Theta_H\alpha^{-1}=\{\theta\alpha^{-1} :
\theta\in\Theta_H\}$. Since $\Theta_H\alpha^{-1}=\Theta_H$ we must
have $\llbracket
(F_a\alpha)(F_b\alpha)\rrbracket_{\sigma}=\llbracket F_a
F_b\rrbracket_{\sigma}$. Thus, since $B^+=F_{a_1}+F_{a_2}+\cdots+F_{a_s}$, we have $\llbracket
(B^+\alpha)(F_b\alpha)-B^+F_b\rrbracket_{\sigma}=\mathbf{0}$, and
hence $\llbracket B^+ B^-\rrbracket_{\sigma}=\mathbf{0}$.

\section{Open problems}
We have shown that $[0.2299,0.2316)$ is an interval of jumps for
$r=3$. If we were able to compute $\pi(\F')$ precisely we could
quite possibly extend this interval below 0.2299. However, as noted
in the introduction, we know that $\pi(\F')>2/9$ so our approach
could never resolve the most important open question in this area:
is $2/9$ a jump?

Indeed the question of whether $2/9$ is a jump for $r=3$ seems
remarkably difficult to resolve. If $2/9$ is a jump then the size of
this jump is very small and so to give a proof along the same lines
as the proof of Theorem \ref{main:thm} would appear to require a
very precise approximation of the Tur\'an density of some unknown
family of 3-graphs. On the other hand the only current technique for
showing a value is \emph{not} a jump is to follow the method of
Frankl and R\"odl \cite{FR}, but this trivially fails for $2/9$ (or
indeed $r!/r^r$ for any $r\geq 3$).

Another obvious open problem is to compute $\pi(K_4^-)$ exactly. It
is likely that improvements over our bound of $0.2871$ could be made
by applying Razborov's method with larger flags or by considering
different types of order $5$. Similarly improved bounds for the
central problem in this area, determining $\pi(K_4^{(3)})$, could
quite probably be found by the use of larger flags.

\section*{Acknowledgments} We would like to thank Dhruv Mubayi for pointing out that Corollary \ref{d:cor} follows from Theorem \ref{k4:thm}.

\end{document}